\documentclass[12pt]{article}

\usepackage{amsthm, amsmath, amssymb,graphicx}

\numberwithin{equation}{section}

\theoremstyle{plain}	     
\newtheorem{thm}{Theorem}[section] 
\newtheorem{cor}[thm]{Corollary}
\newtheorem{lem}[thm]{Lemma}

\theoremstyle{definition}

\theoremstyle{remark} 
\newtheorem{rem}[thm]{Remark}
	
\newcommand{\disp}{\displaystyle}

\newcommand{\vp}{\varphi}

\begin{document}

\title{Applications of generalized trigonometric functions
with two parameters II
\footnote{The work was supported by JSPS KAKENHI Grant Number 17K05336.}}
\author{Shingo Takeuchi \\
Department of Mathematical Sciences\\
Shibaura Institute of Technology
\thanks{307 Fukasaku, Minuma-ku,
Saitama-shi, Saitama 337-8570, Japan. \endgraf
{\it E-mail address\/}: shingo@shibaura-it.ac.jp \endgraf
}}

\date{}

\maketitle

\begin{abstract}
Generalized trigonometric functions (GTFs) are simple generalization
of the classical trigonometric functions. GTFs are deeply related to 
the $p$-Laplacian, which is known as a typical nonlinear differential operator.
Compared to GTFs with one parameter, there are few applications of 
GTFs with two parameters to differential equations.
We will apply GTFs with two parameters 
to studies on the inviscid primitive equations of oceanic and atmospheric dynamics,
new formulas of Gaussian hypergeometric functions, and 
the $L^q$-Lyapunov inequality for 
the one-dimensional $p$-Laplacian.  
\end{abstract}



\section{Introduction}

Let $p,\ q \in (1,\infty)$ be any constants. 
We define $\sin_{p,q}{x}$ by
the inverse function of 
\[
\sin_{p,q}^{-1}{x}:=\int_0^x \frac{dt}{(1-t^q)^{1/p}}
=\frac{1}{q}B_{x^q}\left(\frac{1}{q},\frac{1}{p^*}\right), \quad 0 \leq x \leq 1,
\]
and $\pi_{p,q}$ by
\begin{equation}
\label{eq:pipq}
\pi_{p,q}:=2\sin_{p,q}^{-1}{1}=2\int_0^1 \frac{dt}{(1-t^q)^{1/p}}
=\frac2q B\left(\frac{1}{q},\frac{1}{p^*}\right),
\end{equation}
where $p^*:=p/(p-1)$. Here, $B_x(a,b)$ denotes the incomplete
beta function
\[ B_x(a,b):=\int_0^x t^{a-1}(1-t)^{b-1}\,dt, \quad 0 \leq x \leq 1,\ a,\ b>0,\]
and $B(a,b)$ denotes the beta function 
$$B(a,b):=B_1(a,b), \quad a,\ b>0.$$

Clearly, the function $\sin_{p,q}{x}$ is increasing in $[0,\pi_{p,q}/2]$ onto $[0,1]$.
Since $\sin_{p,q}{x} \in C^1(0,\pi_{p,q}/2)$,
we can define $\cos_{p,q}{x}$ by $\cos_{p,q}{x}:=(d/dx)(\sin_{p,q}{x})$.
In case $p=q$, we denote $\sin_{p,p}{x}$, $\cos_{p,p}{x}$ and $\pi_{p,p}$
briefly by $\sin_p{x}$, $\cos_p{x}$ and $\pi_p$, respectively.
It is obvious that $\sin_{2}{x},\ \cos_{2}{x}$ 
and $\pi_{2}$ are reduced to the ordinary $\sin{x},\ \cos{x}$ and $\pi$,
respectively. This is the reason why these functions and the constant are called
\textit{generalized trigonometric functions} (GTFs) with parameter $(p,q)$
and the \textit{generalized $\pi$}. 
As the trigonometric functions satisfy $\cos^2{x}+\sin^2{x}=1$,
so it is shown that for $x \in [0,\pi_{p,q}/2]$
\begin{equation}
\label{eq:pr}
\cos_{p,q}^p{x}+\sin_{p,q}^q{x}=1.
\end{equation}
In addition, one can see that
$u=\sin_{p,q}{x}$ satisfies the nonlinear differential equation with $p$-Laplacian:
\begin{equation}
\label{eq:pep}
-(|u'|^{p-2}u')'=\frac{q}{p^*}|u|^{q-2}u,
\end{equation}
which is reduced to the equation $-u''=u$ of simple harmonic motion
for $u=\sin{x}$ in case $p=q=2$.

GTFs with one parameter are often used to study problems of
existence, bifurcation and oscillation of solutions of differential equations
related to the $p$-Laplacian
(see \cite{KT2019} and the references given there).
However, there are few applications of GTFs with two parameters
to differential equations, and we can refer only to
Dr\'{a}bek and Man\'{a}sevich \cite{DM1999} and Kobayashi and Takeuchi \cite{KT2019}, 
though GTFs are simple generalization of the classical trigonometric functions.

The present paper is the sequel to \cite{KT2019}
and we will give applications of GTFs with two parameters.

In Section \ref{sec:nbvp}, we will investigate the profiles of positive solutions of 
the following nonlocal boundary value problem.
\begin{equation}
\label{eq:nbvp}
\begin{cases}
\vp'-(\vp')^2+\vp\vp''+\dfrac{2}{H} \disp \int_0^H (\vp'(t))^2\,dt=0,\\
\vp(0)=\vp(H)=0.
\end{cases}
\end{equation} 
This problem was studied in C.\,Cao et al \cite{CINT2015} to investigate the self-similar blowup
for the inviscid primitive equations of oceanic and atmospheric dynamics. 
In \cite[Corollary 1]{KT2019}, it is shown that all the positive solutions of
\eqref{eq:nbvp} are given in terms of GTFs as
\begin{equation}
\label{eq:solnbvp}
\vp_r(x)=\frac{2H}{(2-r)\pi_r}
\sin_r{\left(\frac{\pi_r}{2H}x\right)}
\cos_r^{r-1}{\left(\frac{\pi_r}{2H}x\right)},
\end{equation}
where $r \in (1,2)$ is a free parameter.
Figure \ref{fig:sol_m} shows the graphs of $\vp_r$ for some $r$. 
\begin{figure}[htbp]
\begin{center}
\includegraphics[width=4cm,clip]{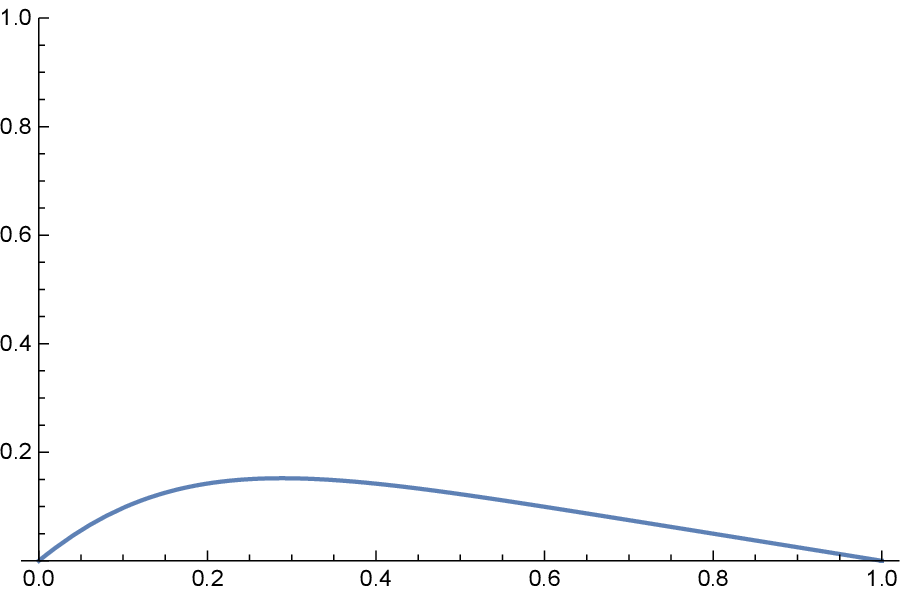}
\includegraphics[width=4cm,clip]{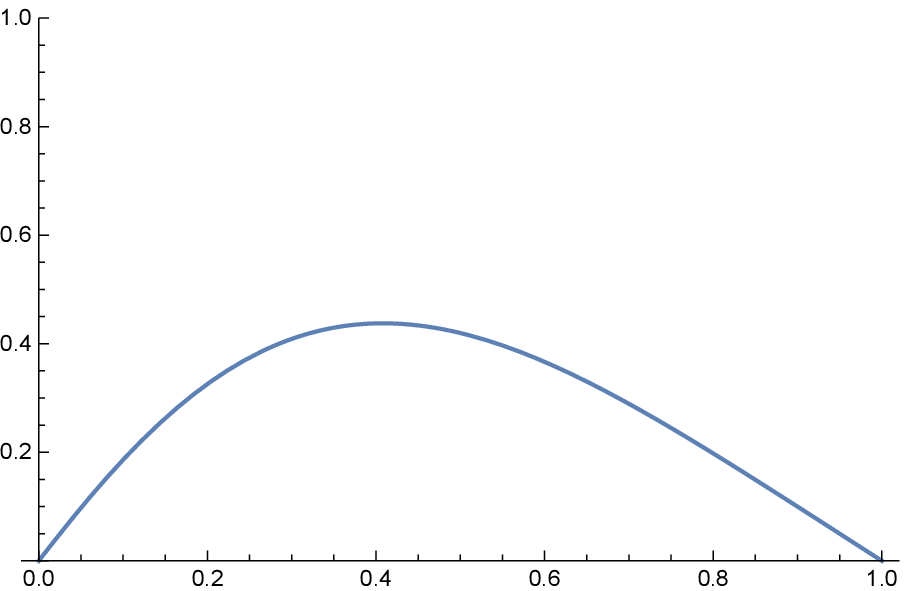}
\includegraphics[width=4cm,clip]{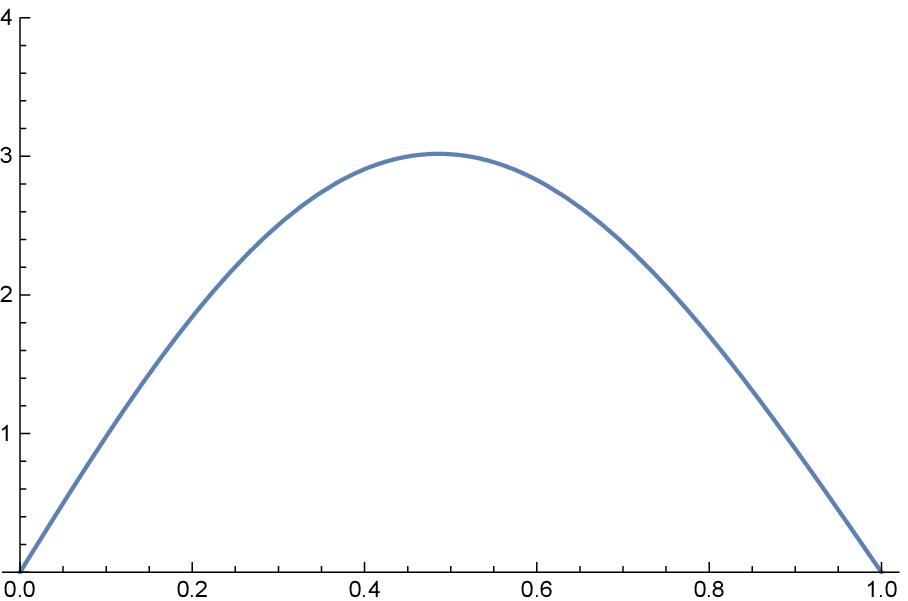}
\caption{Graphs of solutions of \eqref{eq:nbvp} with $H=1$ for $r=1.2,\ 1.5$ and $1.9$.}
\label{fig:sol_m}
\end{center}
\end{figure}

From the graphs in Figure \ref{fig:sol_m}, it is to be expected that 
any positive solution of \eqref{eq:nbvp}
takes the maximum at a point less than $x=H/2$.
Indeed, we can actually prove the following theorem.

\begin{thm}
\label{thm:max}
Any positive solution $\vp_r$ with $r \in (1,2)$ of \eqref{eq:nbvp} has one and only one
extremum 
$$\vp_r(x_r)=\frac{2H}{(2-r)\pi_r}{\frac{1}{r^{1/r}}\frac{1}{(r^*)^{1/r^*}}},$$
which is the maximum, at
$$x_r=\frac{2H}{\pi_r}\sin_r^{-1}{\frac{1}{r^{1/r}}}.$$
Moreover, $x_r<H/2$.
\end{thm}

It is worth pointing out that 
the fact $x_r<H/2$ in Theorem \ref{thm:max} is deduced from the nontrivial inequality 
$$\sin_r{\frac{\pi_r}{4}}>\frac{1}{r^{1/r}}, \quad r \in (1,2),$$
which will be proved in Corollary \ref{cor:max}.
The proof of this inequality relies on the estimate for median of the
beta distribution. We will give such inequalities in the form of two parameters 
(Lemma \ref{lem:maximum} and Corollary \ref{cor:maximum}).  

Section \ref{sec:hgf} establishes the following new formulas of Gaussian hypergeometric 
function $F(a,b;c;x)$ related to GTFs. 
For the definition of $F(a,b;c;x)$, see \eqref{eq:hgf} in Section 3.
\begin{thm}
\label{thm:hgf}
For $p,\ q \in (1,\infty)$ and $x \in (0,1)$,
\begin{align*}
F\left(\frac{1}{q},\frac{1}{p}-1;1+\frac{1}{q};x\right)
&=\frac{q\sin_{p,q}^{-1}{(x^{1/q})}+p^*x^{1/q}(1-x)^{1/p^*}}{(p^*+q)x^{1/q}},\\
F\left(1+\frac{1}{q},\frac{1}{p};2+\frac{1}{q};x\right)
&=\frac{p^*(q+1)(\sin_{p,q}^{-1}{(x^{1/q})}-x^{1/q}(1-x)^{1/p^*})}{(p^*+q)x^{1+1/q}}.
\end{align*}
\end{thm}

In particular, one can find these formulas for $p=q=2$ on the web sites \cite{MFS072303258301}
and \cite{MFS072303288801}, respectively, in the Mathematical Functions Site
by Wolfram Research. Theorem \ref{thm:hgf} gives generalizations of those formulas.

Section \ref{sec:lyapunov} is devoted to the study of 
the $L^q$-Lyapunov inequality for the one-dimensional
$p$-Laplacian. GTFs yield an exact expression to the best constant of the inequality.
Let $p \in (1,\infty)$ and $a \in L^\infty(0,L)$.
Then, we consider the following homogeneous boundary value problem.
\begin{equation}
\label{eq:bvp'}
\begin{cases}
-(|u'|^{p-2}u')'=a(x)|u|^{p-2}u, \quad 0<x<L,\\
u(0)=u(L)=0.
\end{cases}
\end{equation}
A function $u$ is called a solution of \eqref{eq:bvp'} if $u \in W^{1,p}_0(0,L)$ 
satisfies the first equation of \eqref{eq:bvp'} in the weak sense.
We define 
$$\Lambda:=\{a \in L^\infty(0,L): \eqref{eq:bvp'}\ \mbox{has a nontrivial
solution}\}.$$
We denote the $L^q(0,L)$-norm for $q \in [1,\infty]$ by $\|\cdot \|_q$:
for $a \in L^q(0,L)$,
$$\|a\|_q:=
\begin{cases}
\disp \left(\int_0^L |a(x)|^q\,dx\right)^{1/q}, & q \in [1,\infty),\\
\underset{x \in (0,L)}{\operatorname{esssup}}{|a(x)|}, & q=\infty.
\end{cases}$$

In case $q=1$, Elbert \cite[Theorem 6]{E1981} shows that if $a \in \Lambda$, then
\begin{equation}
\label{eq:ineqp_q=1}
\|a\|_1>\frac{2^p}{L^{p-1}}
\end{equation}
and the constant in the right-hand side is optimal.
The inequality \eqref{eq:ineqp_q=1} for $p=2$ is called the \textit{Lyapunov inequality}
(see \cite{CMV2005} and \cite{P2013} for the complete bibliography).

We are interested in the best constant for the $L^q$-norm of $a \in \Lambda$ 
when $q \in (1,\infty)$. In the linear case $p=2$, 
Egorov and Kondratiev \cite{EK1996}, and
Ca\~{n}ada, Montero and Villegas \cite{CMV2005}
give the best constant for $L^q$-norm of $a$
(see also \cite{CV2015} and \cite{P2013}).
Pinasco \cite{P2013} indicates the possibility to extend their results in 
\cite{EK1996} to the nonlinear
case $p \neq 2$ by using GTFs, and gives, however, no expression of the best
constant. By virtue of the idea of \cite{CMV2005} with a result of Dr\'{a}bek and Man\'{a}sevich
\cite{DM1999}, we can obtain the best constant as follows.

\begin{thm}
\label{thm:lyapunov}
Let $p \in (1,\infty)$. Then, for any $a \in \Lambda$,
\begin{equation}
\label{eq:ineqp}
\|a\|_q \geq 
\begin{cases}
\dfrac{2^p(p-1)(q-1)^{p-1+1/q}}{L^{p-1/q}q^{p-1}(pq-1)^{1/q}}
\left(\disp \int_0^{\pi_p/2}\frac{dx}{\sin_p^{1/q}{x}}\right)^p, & q \in (1,\infty),\\[12pt]
(p-1)\left(\dfrac{\pi_p}{L}\right)^p, & q=\infty,
\end{cases}
\end{equation}
where
$$\int_0^{\pi_p/2} \frac{dx}{\sin_p^{1/q}{x}}
=\frac{q^*\pi_{p,pq^*}}{2}
=\frac{1}{p}B\left(\frac{1}{p^*},\frac{1}{pq^*}\right).
$$
Moreover, the constants of \eqref{eq:ineqp} are optimal and attained by 
\begin{equation*}
a(x)=
\begin{cases}
(p-1)q^*\left(\dfrac{\pi_{p,pq^*}}{L}\right)^p \sin_{p,pq^*}^{p/(q-1)}
{\left(\dfrac{\pi_{p,pq^*}}{L}x\right)}, & q \in (1,\infty),\\[12pt]
(p-1)\left( \dfrac{\pi_p}{L} \right)^p, & q=\infty.
\end{cases}
\end{equation*}
\end{thm}

In case $p=2$, the constants in the right-hand side of \eqref{eq:ineqp} are same as 
in \cite[Theorem 2.1]{CMV2005}.

This paper is organized as follows. Section \ref{sec:nbvp} deals with the profiles
of positive solutions of the nonlocal boundary value problem \eqref{eq:nbvp} and
we prove Theorem \ref{thm:max}.
Section \ref{sec:hgf} provides formulas of Gaussian hypergeometric functions related
to GTFs and we show Theorem \ref{thm:hgf}.
Section \ref{sec:lyapunov} is intended to obtain the best constant of $L^q$-Lyapunov
inequality for the one-dimensional $p$-Laplacian and to give the proof of Theorem \ref{thm:lyapunov}.


\section{Proof of Theorem \ref{thm:max}}
\label{sec:nbvp}

To show (the latter half of) Theorem \ref{thm:max}, 
the following lemma is crucial. 

\begin{lem}
\label{lem:maximum}
If $p^*>q>1$, then
\begin{equation}
\label{eq:pi/4}
\left(\frac{p^*}{p^*+q}\right)^{1/q}<\sin_{p,q}{\frac{\pi_{p,q}}{4}};
\end{equation}
if $p^*=q>1$, then
\begin{equation}
\label{eq:pi/4=}
\sin_{p,p^*}{\frac{\pi_{p,p^*}}{4}}=\frac{1}{2^{1/p^*}};
\end{equation}
and if $q>p^*>1$, then
$$\sin_{p,q}{\frac{\pi_{p,q}}{4}}<\left(\frac{p^*}{p^*+q}\right)^{1/q}.$$
\end{lem}

\begin{proof}
Let $I_x(a,b)$ denote the regularized incomplete beta function 
$$I_x(a,b):=\frac{B_x(a,b)}{B(a,b)}, \quad 0 \leq x \leq 1,\ a,\ b>0.$$
It is easily seen that $I_x$ satisfies
\begin{equation}
\label{eq:I_x}
I_x(a,b)=1-I_{1-x}(b,a)
\end{equation}
(see for instance \cite[6.6.3 in p.\,263]{AS1964} and \cite[8.17.4 in p.\,183]{OLBC2010}).

Let 
$$X_{p,q}:=\frac{p}{p+q}.$$
From the definition of $\sin_{p,q}^{-1}{x}$, 
\begin{equation}
\label{eq:sinpq}
\sin_{p,q}^{-1}{(X_{p^*,q}^{1/q})}=\frac{1}{q}B_{X_{p^*,q}}\left(\frac{1}{q},\frac{1}{p^*}\right)
=\frac{\pi_{p,q}}{2}I_{X_{p^*,q}}\left(\frac{1}{q},\frac{1}{p^*}\right).
\end{equation}
Following Payton, Young and Young \cite{PYY1989} and setting
$$s=\log{\frac{p^*(1-t)}{qt}},$$
we have
\begin{align}
&I_{X_{p^*,q}}\left(\frac{1}{q},\frac{1}{p^*}\right) \notag \\
& \quad =\frac{1}{B(1/q,1/p^*)}\int_{\infty}^0 
\left(\frac{p^*}{p^*+qe^s}\right)^{1/q-1}
\left(\frac{qe^s}{p^*+qe^s}\right)^{1/p^*-1}
\frac{-p^*qe^s}{(p^*+qe^s)^2}\,ds \notag \\
& \quad =\frac{(p^*)^{1/q}q^{1/p^*}}{B(1/q,1/p^*)}
\int_0^\infty \frac{e^{s/p^*}}{(p^*+qe^s)^{1/p^*+1/q}}\,ds.
\label{eq:x}
\end{align}
Moreover, interchanging $p^*$ into $q$, we obtain
\begin{align}
I_{X_{q,p^*}}\left(\frac{1}{p^*},\frac{1}{q}\right)
&=\frac{q^{1/p^*}(p^*)^{1/q}}{B(1/p^*,1/q)}
\int_0^\infty \frac{e^{s/q}}{(q+p^*e^s)^{1/q+1/p^*}}\,ds \notag \\
&=\frac{(p^*)^{1/q}q^{1/p^*}}{B(1/q,1/p^*)}
\int_0^\infty \frac{e^{-s/p^*}}{(p^*+qe^{-s})^{1/p^*+1/q}}\,ds.
\label{eq:1-x}
\end{align}

Consider the case $p^*>q>1$.
In this case, we can see that for $s>0$,
\begin{equation}
\label{eq:integrands}
\frac{e^{s/p^*}}{(p^*+qe^s)^{1/p^*+1/q}}<\frac{e^{-s/p^*}}{(p^*+qe^{-s})^{1/p^*+1/q}}.
\end{equation}
Indeed, it is equivalent to
$$\frac{\sinh{(X_{q,p^*}s)}}{X_{q,p^*}s}<\frac{\sinh{(X_{p^*,q}s)}}{X_{p^*,q}s},$$
which holds true since $\sinh{x}/x$ is strictly increasing. It follows from \eqref{eq:x}--\eqref{eq:integrands}
that
$$I_{X_{p^*,q}}\left(\frac{1}{q},\frac{1}{p^*}\right)
<I_{X_{q,p^*}}\left(\frac{1}{p^*},\frac{1}{q}\right).$$
Since $X_{p^*,q}+X_{q,p^*}=1$ and \eqref{eq:I_x}, we have
$$I_{X_{p^*,q}}\left(\frac{1}{q},\frac{1}{p^*}\right)
=1-I_{X_{q,p^*}}\left(\frac{1}{p^*},\frac{1}{q}\right)
<1-I_{X_{p^*,q}}\left(\frac{1}{q},\frac{1}{p^*}\right);$$
so that
\begin{equation}
\label{eq:betadistribution}
I_{X_{p^*,q}}\left(\frac{1}{q},\frac{1}{p^*}\right)<\frac12.
\end{equation}
Therefore,  by \eqref{eq:sinpq},
$$\sin_{p,q}^{-1}{(X_{p^*,q}^{1/q})}<\frac{\pi_{p,q}}{4},$$
and \eqref{eq:pi/4} is proved. The remaining cases also follow in a similar way.
\end{proof}

\begin{rem}
(i) The equality \eqref{eq:pi/4=} is also obtained in \cite[Lemma 2.1]{T2016b}.

(ii) The inequality \eqref{eq:betadistribution} means that $X_{p^*,q}$ is less than the median of beta
distribution with parameters $1/q$ and $1/p^*$.
\end{rem}

\begin{cor}
\label{cor:max}
If $r \in (1,2)$, then
\begin{equation}
\label{eq:1<r<2}
\frac{1}{r^{1/r}}<\sin_r{\frac{\pi_r}{4}};
\end{equation}
if $r=2$, then
$$\sin_2{\frac{\pi_2}{4}}=\frac{1}{\sqrt{2}};$$
and if $r \in (2,\infty)$, then
$$\sin_r{\frac{\pi_r}{4}}<\frac{1}{r^{1/r}}.$$
\end{cor}

\begin{proof}
Let $r \in (1,2)$. Then $r^*>r>1$, and hence \eqref{eq:pi/4} with $p=q=r$,
i.e. \eqref{eq:1<r<2} holds true. 
The remaining parts also follow in a similar way.
\end{proof}

We are now in a position to show Theorem \ref{thm:max}.

\begin{proof}[Proof of Theorem \ref{thm:max}]
Differentiating \eqref{eq:solnbvp} with using \eqref{eq:pep}, we have
\begin{align*}
\vp_r'(x)
&=\frac{1}{2-r}\left(-(r-1)\sin_r^r{\left(\frac{\pi_r}{2H}x\right)}
+\cos_r^r{\left(\frac{\pi_r}{2H}x\right)}\right) 
\\
&=\frac{1}{2-r}\left(1-r\sin_r^r{\left(\frac{\pi_r}{2H}x\right)}\right).
\end{align*}
Thus, $\vp_r$ has the maximum 
$$\vp_r(x_r)=\frac{2H}{(2-r)\pi_r}{\frac{1}{r^{1/r}}\frac{1}{(r^*)^{1/r^*}}}$$
only at
$$x=x_r:=\frac{2H}{\pi_r}\sin_r^{-1}{\frac{1}{r^{1/r}}}.$$
Moreover, since $r \in (1,2)$, by \eqref{eq:1<r<2} of Corollary \ref{cor:max}, 
$$x_r<\frac{2H}{\pi_r}\cdot \frac{\pi_r}{4}=\frac{H}{2},$$
and the proof is complete.
\end{proof}

\begin{rem}
Observing \eqref{eq:nbvp} directly, one can show the facts:
$\vp_r$ has no local minimum in $(0,1)$;
and $\vp_r$ is asymmetric with respect to $x=H/2$.
However, it seems to be difficult to prove $x_r<H/2$ in this way.
\end{rem}

In \cite[Theorem 2.1]{KT2019}, 
the authors also study the following problem to solve \eqref{eq:nbvp}.
\begin{equation}
\label{eq:gnbvp}
\begin{cases}
(p-q)u'-pq(u')^2+(p+q)uu''+1=0,\\
u(0)=u(H)=0.
\end{cases}
\end{equation}
The positive solution of \eqref{eq:gnbvp} is
$$u=\frac{2H}{q\pi_{p^*,q}} \cos_{p^*,q}^{p^*-1}
{\left(\frac{\pi_{p^*,q}}{2H}x\right)}\sin_{p^*,q}{\left(\frac{\pi_{p^*,q}}{2H}x\right)}.$$
As in the proof of Theorem \ref{thm:max}, with the aid of 
Lemma \ref{lem:maximum}, we can show the following result.
\begin{cor}
\label{cor:maximum}
Any positive solution $u$ with $p,\ q \in (1,\infty)$ of \eqref{eq:gnbvp} has one and only one
extremum 
$$u(x_{p,q})=\frac{2H}{q\pi_{p^*,q}} \left(\frac{q}{p+q}\right)^{1/p}
\left(\frac{p}{p+q}\right)^{1/q},$$
which is the maximum, at
$$x_{p,q}=\frac{2H}{\pi_{p^*,q}}\sin_{p^*,q}^{-1}{\left(\frac{p}{p+q}\right)^{1/q}}$$
Moreover, $x_{p,q}<H/2$ if $p>q>1$; $x_{p,q}=H/2$ if $p=q>1$; $x_{p,q}>H/2$ if $q>p>1$.
\end{cor}


\section{Proof of Theorem \ref{thm:hgf}}
\label{sec:hgf}

For $a,\ b \in \mathbb{R},\ c \neq 0,-1,-2,\cdots$,
a Gaussian hypergeometric function is defined as
\begin{equation}
\label{eq:hgf}
F(a,b;c;x):=\sum_{n=0}^\infty \frac{(a)_n(b)_n}{(c)_n}\frac{x^n}{n!},\quad |x|<1,
\end{equation}
where
$$(a)_n:=\frac{\Gamma(a+n)}{\Gamma(a)}=a(a+1)(a+2)\cdots(a+n-1),\ (a)_0:=1.$$

\begin{lem}
\label{lem:integral}
For $p,\ q \in (1,\infty)$ and $x \in [0,\pi_{p,q}/2]$, 
\begin{align}
\int_0^x \cos_{p,q}^p{x}\,dx
&=\frac{qx+p^*\sin_{p,q}{x}\cos_{p,q}^{p-1}{x}}{p^*+q}, \notag \\
\int_0^x \sin_{p,q}^q{x}\,dx
&=\frac{p^*x-p^*\sin_{p,q}{x}\cos_{p,q}^{p-1}{x}}{p^*+q}. \label{eq:sin}
\end{align}
\end{lem}

\begin{proof}
Set
$$I=\int_0^x \cos_{p,q}^p{x}\,dx, \quad
J=\int_0^x \sin_{p,q}^q{x}\,dx.$$
By \eqref{eq:pr}, it is easy to see that
\begin{equation}
\label{eq:IJ1}
I+J=x.
\end{equation}
Integrating $J$ by parts and using \eqref{eq:pep}, we obtain
$$J=\int_0^x \sin_{p,q}{x}\sin_{p,q}^{q-1}{x}\,dx
=\left[\sin_{p,q}{x} \left(-\frac{p^*}{q}\cos_{p,q}^{p-1}{x}\right)\right]_0^x
+\frac{p^*}{q}I;$$
thus,
\begin{equation}
\label{eq:IJ2}
J=-\frac{p^*}{q}\sin_{p,q}{x}\cos_{p,q}^{p-1}{x}+\frac{p^*}{q}I.
\end{equation}
From \eqref{eq:IJ1} and \eqref{eq:IJ2}, we obtain the assertion.
\end{proof}

\begin{cor}
Let $r \in (1,\infty)$. For $x \in [0,\pi_r/2]$,
\begin{align*}
\int_0^x \cos_r^r{x}\,dx
&=\frac{x}{r^*}+\frac{\sin_r{x}\cos_r^{r-1}{x}}{r},\\
\int_0^x \sin_r^r{x}\,dx
&=\frac{x}{r}-\frac{\sin_r{x}\cos_r^{r-1}{x}}{r};
\end{align*}
for $x \in [0,\pi_{r^*,r}/2]=[0,\pi_{2,r}/2^{2/r}]$,
\begin{align*}
\int_0^x \cos_{r^*,r}^r{x}\,dx
&=\frac{x}{2}+\frac{\sin_{2,r}{(2^{2/r}x)}}{2^{1+2/r}}\\
\int_0^x \sin_{r^*,r}^r{x}\,dx
&=\frac{x}{2}-\frac{\sin_{2,r}{(2^{2/r}x)}}{2^{1+2/r}}.
\end{align*}
\end{cor}

\begin{proof}
The former half is Lemma \ref{lem:integral} for $p=q=r$
(this was proved by Bushell and Edmunds \cite[Proposition 2.6]{BE2012}).
For the latter half,  taking $p^*=q=r$ in Lemma \ref{lem:integral}
and using the multiple-angle 
formula \cite[Theorem 1.1]{T2016b}: for $x \in [0,\pi_{r^*,r}/2]=[0,\pi_{2,r}/2^{2/r}]$
$$\sin_{2,r}{(2^{2/r}x)}=2^{2/r}\sin_{r^*,r}{x}\cos_{r^*,r}^{r^*-1}{x},$$
we immediately conclude the assertion.
\end{proof}

We proceed to show Theorem \ref{thm:hgf}.

\begin{proof}[Proof of Theorem \ref{thm:hgf}]
Let $I,\ J$ be the integrals in the proof of Lemma \ref{lem:integral}.
The integral formula \cite[(14) in Theorem 3.1]{KT2019} gives: for $x \in (0,\pi_{p,q}/2)$
\begin{align*}
I&=\sin_{p,q}{x}\,F\left(\frac{1}{q},\frac{1}{p}-1;1+\frac{1}{q};\sin_{p,q}^q{x}\right),\\
J&=\frac{1}{q+1}\sin_{p,q}^{q+1}{x}\,F\left(1+\frac{1}{q},\frac{1}{p};2+\frac{1}{q};\sin_{p,q}^q{x}\right).
\end{align*}
Combining them with Lemma \ref{lem:integral}, we have 
\begin{align}
F\left(\frac{1}{q},\frac{1}{p}-1;1+\frac{1}{q};\sin_{p,q}^q{x}\right)
&=\frac{qx+p^*\sin_{p,q}{x}\cos_{p,q}^{p-1}{x}}{(p^*+q)\sin_{p,q}{x}}, \label{eq:1/q}\\
F\left(1+\frac{1}{q},\frac{1}{p};2+\frac{1}{q};\sin_{p,q}^q{x}\right)
&=\frac{p^*(q+1)(x-\sin_{p,q}{x}\cos_{p,q}^{p-1}{x})}{(p^*+q)\sin_{p,q}^{q+1}{x}},
\label{eq:1+1/q}
\end{align}
which imply the assertion.
In fact, \eqref{eq:1+1/q} is obtained also by differentiating both sides of \eqref{eq:1/q},
since $(d/dx)F(a,b;c;x)=(ab/c)F(a+1,b+1;c+1;x)$. 
\end{proof}


\section{Proof of Theorem \ref{thm:lyapunov}}
\label{sec:lyapunov}

Let $p \in (1,\infty)$ and $a \in L^\infty(0,L)$.
Then, we consider \eqref{eq:bvp'}, i.e.,
the following homogeneous boundary value problem.
\begin{equation}
\label{eq:bvp}
\begin{cases}
-(\phi(u'))'=a(x)\phi(u), \quad 0<x<L,\\
u(0)=u(L)=0,
\end{cases}
\end{equation}
where $\phi(s):=|s|^{p-2}s$ for $s \neq 0$; $=0$ for $s=0$.
Recall 
$$\Lambda:=\{a \in L^\infty(0,L): \eqref{eq:bvp}\ \mbox{has a nontrivial
solution}\}.$$

\begin{proof}[Proof of Theorem \ref{thm:lyapunov}]
First of all, we will show the case $q=\infty$. 
Let $a \in \Lambda$ and $u$ be the nontrivial solution of \eqref{eq:bvp}.
Then we have
$$\int_0^L |u'|^p\,dx=\int_0^L a(x)|u|^p\,dx \leq \|a\|_\infty \int_0^L |u|^p\,dx.$$
Therefore,
\begin{equation}
\label{eq:ainfty}
\|a\|_\infty \geq \frac{\disp \int_0^L |u'|^p\,dx}{\disp \int_0^L |u|^p\,dx}
\geq \lambda_0:=(p-1)\left(\frac{\pi_p}{L}\right)^p,
\end{equation}
where $\lambda_0$ is the first eigenvalue of $p$-Laplacian (see e.g. \cite[Theorem A.4]{P2013}).

Then, the constant function 
$$a_\infty(x):=\lambda_0=(p-1)\left(\frac{\pi_p}{L}\right)^p$$
is an element of $\Lambda$ and attains the equalities of \eqref{eq:ainfty}. 
Indeed, \eqref{eq:bvp} for $a=a_\infty$ has 
the nontrivial solution $u=\sin_p{(\pi_p x/L)}$, the eigenfunction corresponding 
to $\lambda_0$. 

Next, we consider the case $q \in (1,\infty)$.
Let $X=W^{1,p}_0(0,L) \setminus \{0\}$, 
$a \in \Lambda$ and $u$ be the nontrivial solution of \eqref{eq:bvp}.
Then we have, by H\"{o}lder's inequality,
$$\int_0^L |u'|^p\,dx=\int_0^L a|u|^p\,dx \leq \|a\|_q 
\left(\int_0^L |u|^{pq^*}\,dx\right)^{1/q^*}.$$
Therefore, defining the functional $J_q: X \to \mathbb{R}$ as
$$J_q(v):=\frac{\disp \int_0^L |v'|^p\,dx}{\left(\disp 
\int_0^L |v|^{pq^*}\,dx\right)^{1/q^*}}$$
and its infimum
\begin{equation}
\label{eq:mqinf}
m_q:=\inf_{v \in X} J_q(v),
\end{equation}
we obtain
\begin{equation}
\label{eq:aq}
\|a\|_q \geq J_q(u) \geq m_q.
\end{equation}

It follows from a standard compactness argument and Lagrange's multiplier technique
(e.g. \cite[Theorem 2 in p.489]{E2010}) that $m_q$ is attained 
by the minimizer $u_q \in X$ satisfying
\begin{equation}
\label{eq:2.14'}
\begin{cases}
(\phi(u_q'))'+A_q(u_q)|u_q|^{pq^*-2}u_q=0,\\
u_q(0)=u_q(L)=0,
\end{cases}
\end{equation}
where
\begin{equation}
\label{eq:A_q}
A_q(u_q)=m_q \left(\int_0^L |u_q|^{pq^*}\,dx\right)^{-1/q}.
\end{equation}
In other words, $u_q$ satisfies
\begin{equation}
\label{eq:2.14}
\begin{cases}
(\phi(u_q'))'+a_q(x)\phi(u_q)=0,\\
u_q(0)=u_q(L)=0,
\end{cases}
\end{equation}
where 
\begin{equation}
\label{eq:a_q}
a_q(x):=A_q(u_q)|u_q(x)|^{p/(q-1)}.
\end{equation}

Then, the function $a_q$ is an element of $\Lambda$ and attains the equalities of \eqref{eq:aq}. 
Indeed, \eqref{eq:2.14} implies that \eqref{eq:bvp} for $a=a_q$ has 
the nontrivial solution $u_q$, and an easy calculation yields
$\|a_q\|_q=m_q$. 

Finally we will evaluate $m_q$ and give the expression of function $a_q$.
Since solution $u_q$ of \eqref{eq:2.14'} can be taken to be nonnegative, 
we can write 
\begin{gather}
A_q(u_q)
=\frac{pq^*}{p^*}\left(\frac{\pi_{p,pq^*}}{L}\right)^p R^{p-pq^*}, \label{eq:aaa}\\
u_q
=R\sin_{p,pq^*}{\left(\frac{\pi_{p,pq^*}}{L}x\right)} \label{eq:bbb}
\end{gather}
for some $R>0$ (cf. \cite{DM1999} and \cite[Theorem 2.1]{T2012}).
Substituting \eqref{eq:aaa} and \eqref{eq:bbb} into \eqref{eq:A_q},
we obtain
\begin{align*}
m_q
&=\frac{pq^*}{p^*}\left(\frac{\pi_{p,pq^*}}{L}\right)^p R^{p-pq^*}
\cdot R^{pq^*/q}
\left(\int_0^L\left|\sin_{p,pq^*}{\left(\frac{\pi_{p,pq^*}}{L}x\right)}
\right|^{pq^*}\,dx\right)^{1/q}\\
&=\frac{pq^*}{p^*}\left(\frac{\pi_{p,pq^*}}{L}\right)^p
\left(\frac{L}{\pi_{p,pq^*}}\right)^{1/q}
\left( 2 \int_0^{\pi_{p,pq^*}/2} \sin_{p,pq^*}^{pq^*}{t}\,dt\right)^{1/q}\\
&=\frac{pq^*\pi_{p,pq^*}^p}{L^{p-1/q}(p^*)^{1/q^*}(p^*+pq^*)^{1/q}}.
\end{align*}
Here, we used \eqref{eq:sin} for the integral calculation.  
Moreover, letting $t^{q^*}=\sin_p{x}$, we have
$$\pi_{p,pq^*}=2\int_0^1\frac{dt}{(1-t^{pq^*})^{1/p}}=\frac{2}{q^*}\int_0^{\pi_p/2}
\frac{dx}{\sin_p^{1/q}{x}}.$$
Thus, we conclude that
\begin{equation*}
m_q=\dfrac{2^p(p-1)(q-1)^{p-1+1/q}}{L^{p-1/q}q^{p-1}(pq-1)^{1/q}}
\left(\disp \int_0^{\pi_p/2}\frac{dx}{\sin_p^{1/q}{x}}\right)^p.
\end{equation*}
Function $a_q$ follows immediately from \eqref{eq:a_q} with \eqref{eq:aaa} and \eqref{eq:bbb}. 
\end{proof}

\begin{rem}
In a similar way to the proof of \cite[Lamma 2.9]{CMV2005}, 
it is possible to show that $\lim_{q \to 1+0}m_q=2^p/L^{p-1}$ 
and $\lim_{q \to \infty}m_q=(p-1)(\pi_p/L)^p$.
These constants are the best constants of $L^q$-Lyapunov inequalities 
\eqref{eq:ineqp_q=1} for $q=1$ and \eqref{eq:ineqp} for $q=\infty$, respectively.
\end{rem}

\begin{rem}
From \eqref{eq:mqinf}, we obtain 
the Sobolev-Poincar\'e inequality with best constant. Indeed, 
we obtain that $J_q(v) \geq m_q$, $q \in (1,\infty)$, 
for all $v \in X$. Letting $pq^*=r$, we see that for all $v \in W^{1,p}_0(0,L)$,
$$\|v\|_{r} \leq 
\frac{\left(1+\frac{p^*}{r}\right)^{1/p}}{\left(1+\frac{r}{p^*}\right)^{1/r}}
\frac{L^{1/p^*+1/r}}{\pi_{p,r}} \|v'\|_p, \quad r \in (p,\infty).$$
We emphasize that this result was already known
(see e.g. \cite[Theorem 5.1]{DM1999}, where the definition of $\pi_{p,r}$
in \cite{DM1999} is slightly different from \eqref{eq:pipq}). 
\end{rem}



\end{document}